\documentclass[11pt]{article}

\usepackage{fullpage,amsthm,amsmath,amssymb,amsbsy,amsfonts}

\usepackage{graphicx}

\title{\bf Sharp pointwise estimates for the gradients of solutions  to linear parabolic second order equation in the layer}

\author{\sc{Gershon Kresin$^a\!\!$}
\thanks{Corresponding author. E-mail: kresin@ariel.ac.il}$\;\;$
 and \sc{Vladimir Maz'ya$^b$}
\thanks{E-mail: vladimir.mazya@liu.se}$\;\;$ 
\\ \\
{\it{$^a$Department of Mathematics, Ariel University, Ariel 40700, Israel}}\\
{\it{$^b$Department of Mathematical Sciences, University of Liverpool,
M$\&$O Building, Liverpool,}}\\ 
{\it{L69 3BX, UK; Department of Mathematics, Link\"oping University,SE-58183 Link\"oping, }}\\
{\it{{\hskip -19mm}Sweden; }}
{\it{RUDN University, 6 Miklukho-Maklay St., Moscow, 117198, Russia}}\\
}
{ \date\ }

\numberwithin{equation}{section}
\newtheorem{lemma}{Lemma}
\newtheorem{theorem}{Theorem}
\newtheorem{proposition}[theorem]{Proposition}

\newenvironment{remark}{{\bf Remark}}

\newcommand{\bs}{\boldsymbol}
\newcommand{\bv}{\bs |\mskip-5mu \bs |\mskip-5mu \bs |}  
\newcommand{\nl}{\lVert}
\newcommand{\nr}{\rVert}

\begin{document}
\maketitle
%\LARGE
\large

\vspace{10mm}

%%%%%%%%%%%%%%%%%%%%%%%%%%%%%%%%%%%%%%%%%%%%%%%%%%%%%%%%%%%%%%%%%%
{\bf Abstract.} We deal with solutions of the Cauchy problem to  linear both homogeneous and nonhomogeneous parabolic second order equations with real constant coefficients in the layer ${\mathbb R}^{n+1}_T={\mathbb R}^n\times (0, T)$, where $n\geq 1$ and $T<\infty$. The homogeneous equation is considered with initial data in $L^p({\mathbb R}^n)$, $1\leq p \leq \infty $. For the nonhomogeneous equation we suppose that initial function is equal to zero and the function in the right-hand side belongs to $f\in L^p({\mathbb R}^{n+1}_T)\cap C^\alpha \big (\overline{{\mathbb R}^{n+1}_T} \big ) $ , $p>n+2$ and $\alpha \in (0, 1)$. Explicit formulas for the sharp coefficients in pointwise estimates for the length of the gradient to solutions to these problems are obtained.
%%%%%%%%%%%%%%%%%%%%%%%%%%%%%%%%%%%%%%%%%%%%%%%%%%%%%%%%%%%%%%%%%%
\\
%%%%%%%%%%%%%%%%%%%%%%%%%%%%%%%%%%%%%%%%%%%%%%%%%%%%
\section{Introduction}\label{S_1}
%%%%%%%%%%%%%%%%%%%%%%%%%%%%%%%%%%%%%%%%%%%%%%%%%%%%
Linear parabolic equations of the second order find numerous applications. They model convective heat and mass transfer processes (e.g. \cite{PZ}, Sect. 2.5.5), diffusion of gases (e.g. \cite{TS}, Ch. VI, Sect. 2) etc.

It is known that bounded solutions $u(x, t)$ of the linear parabolic equation of the second order with constant real coefficients
\begin{equation} \label{Eq_0.1}
\frac{\partial u}{\partial t}= \sum_{j,k=1}^n a_{jk}\frac{\partial^2 u}{\partial x_j \partial x_k}+
\sum_{j=1}^n b_{j}\frac{\partial u}{\partial x_j}+c u\;,
\end{equation}
where $c\leq 0$, $(x, t) \in{\mathbb R}^{n+1}_T={\mathbb R}^n\times (0, T)$, satisfy the weak maximum principle 
$$
\sup_{{\mathbb R}^{n+1}_T}|u|=\sup_{y \in {\mathbb R}^n}|u(y, 0)|\;.
$$
In the general case of any real $c$, the following pointwise estimate with the best possible coefficient holds
\begin{equation} \label{Eq_0.1A}
|u(x, t)|\leq e^{ct}\sup_{y\in {\mathbb R}^n}|u(y, 0)|\;,
\end{equation}
where $(x, t)$ is an arbitrary point of the layer ${\mathbb R}^{n+1}_T$, $T< \infty $. As for the sharp pointwise estimate for
$|\nabla_x u(x,t)|$, where $|\cdot |$ means the Euclidean length of a vector in ${\mathbb R}^n$, it was unknown  even for the heat equation.
Here and elsewhere we say that the estimate  is sharp if the coefficient in front of the norm in the  majorant part of the inequality can not be diminished. 

In the present paper, we extend our study of sharp estimates 
for solutions to the Laplace, Lam\'e, Stokes and heat equations (see \cite{KM1}-\cite{KM3} and references there) as well as for analytic functions \cite{KM}. Here we obtain sharp pointwise estimates for $|\nabla_x u(x,t)|$, 
where $u$ solves the Cauchy problem for linear parabolic equations of the second order with constant real coefficients in ${\mathbb R}^{n+1}_T$, where $T<\infty$ and $n\geq 1$. 

Now, we will describe our main results.  
Section \ref{S_2} is devoted to explicit formulas for solutions of the Cauchy problem in ${\mathbb R}^{n+1}_T$ for both homogeneous 
and nonhomogeneous parabolic equations with constant real coefficients. 

In Section \ref{S_3}, we consider a solution of the Cauchy problem 
\begin{eqnarray} \label{H}
\left\{\begin{array}{ll}
\displaystyle{\frac{\partial u}{\partial t}= \sum_{j,k=1}^n a_{jk}\frac{\partial^2 u}{\partial x_j \partial x_k}}+\sum_{j=1}^n b_{j}\frac{\partial u}{\partial x_j}
+c u& \quad{\rm in}\; {\mathbb R}^{n+1}_T, \\
       \\
\displaystyle{ u\big |_{t=0}=\varphi }\;,  
\end{array}\right .
\end{eqnarray}
where $A=((a_{jk}))$ is a symmetric positive definite matrix of order $n$ and $\varphi\in L^p({\mathbb R}^n)$, $p \in [1, \infty]$.
The norm $\nl \cdot \nr_{p}$ in the space $L^p({\mathbb R}^n)$ is defined by
$$
\nl \varphi\nr_{p}=\left \{\int _ {{\mathbb R}^n} |  \varphi(x) |^p dx \right \}^{1/p}
$$
for $1\leq  p< \infty $, and 
$$
\nl \varphi \nr_{\infty} =\mbox{ess}\;\sup \{ | \varphi(x) |: x \in {{\mathbb R}^n } \}\;. 
$$ 

The explicit formula for the sharp coefficient 
\begin{equation} \label{Eq_1.3}
{\mathcal K}_{p, \ell}(t)=\frac{\big |A^{-1/2}\ell \big |}
{\big \{ 2^{n}\pi^{(n+p-1)/2}\det A^{1/2} \big \}^{1/p}}
\left \{ \frac{\Gamma \left (\frac{p'+1}{2} \right )} {p'^{(n+p')/2}} \right \}^{1/p'}\frac{e^{ct}}{t^{(n+p)/(2p)}}
\end{equation}
in the inequality 
\begin{equation} \label{Eq_1.3A}
\left | \frac{\partial u}{ \partial {\ell}}(x, t) \right |\leq {\mathcal K}_{p, \ell}(t) \nl \varphi \nr_{p}
\end{equation}
for solutions of the Cauchy problem (\ref{H}) is obtained, where $p^{-1}+p'^{-1}=1$, $(x, t)$ is an arbitrary point in the layer ${\mathbb R}^{n+1}_T$ and $\ell$ stands for a unit $n$-dimensional vector. As a consequence of (\ref{Eq_1.3}), the sharp coefficient 
\begin{equation} \label{Eq_1.3AB}
{\mathcal K}_p(t)=\max_{|\ell |=1}{\mathcal K}_{p, \ell}(t)
\end{equation}
in the inequality
\begin{equation} \label{Eq_1.3ABC}
\left | \nabla_x u(x, t) \right |\leq {\mathcal K}_p(t) \nl \varphi \nr_{p}
\end{equation}
is found. 
In particular, 
\begin{equation} \label{Eq_1.4}
{\mathcal K}_\infty(t)=\frac{\bv A^{-1/2} \bv }{\sqrt{\pi }}\;\frac{e^{ct}}{t^{1/2}}\;,
\end{equation}
where $\bv B \bv=\max_{|\ell|=1}|B\ell |$ is 
the spectral norm of the real-valued matrix $B$. It is known (e.g. \cite{LA}, sect. 6.3) that $\bv B \bv=\lambda ^{1/2}_B$, where 
$\lambda _B$ is the spectral radius of the matrix $B'B$. Here the
symbol $'$ denotes passage to the transposed matrix. 

As a special case of (\ref{Eq_1.4}) one has
$$
{\mathcal K}_\infty(t)=\frac{1}{\sqrt{a\pi }}\;\frac{e^{ct}}{t^{1/2}}
$$
for $A=aI$, where $I$ is the unit matrix of order $n$. 

In Section \ref{S_4} we consider a solution of the Cauchy problem 
\begin{eqnarray}  \label{NH}
\left\{\begin{array}{ll}
\displaystyle{\frac{\partial u}{\partial t}= 
\sum_{j,k=1}^n a_{jk}\frac{\partial^2 u}{\partial x_j \partial x_k}}+\sum_{j=1}^n b_{j}\frac{\partial u}{\partial x_j}
+c u+f(x,t)& \quad{\rm in}\; {\mathbb R}^{n+1}_T, \\
       \\
\displaystyle{ u\big |_{t=0}=0}\;,  
\end{array}\right . 
\end{eqnarray}
where $A=((a_{jk}))$ is a symmetric positive definite matrix of order $n$ and 
$f\in L^p({\mathbb R}^{n+1}_T)\cap C^\alpha\big (\overline{{\mathbb R}^{n+1}_T} \big ) $ with $p>n+2$ and $\alpha \in (0, 1)$. 
By $C^\alpha \big (\overline{{\mathbb R}^{n+1}_T} \big )$ we denote the space of functions $f(x, t)$ which are continuous and bounded in $\overline{{\mathbb R}^{n+1}_T}$ and locally H\"older
continuous with exponent $\alpha$ in $x\in {\mathbb R} ^n$, uniformly with respect to $t\in[0, T]$.
The space $L^p\big ({\mathbb R}^{n+1}_T\big )$ is endowed with the norm
$$
\nl f \nr_{p, T}=\left \{ \int_0^T\int _ {{\mathbb R}^n } \big | f(x, \tau )\big |^p dx d\tau\right \}^{1/p}
$$
for $1\leq  p< \infty $, and 
$$
\nl f \nr_{\infty,T} =\mbox{ess}\;\sup \{ | f(x, \tau) |:
x \in {{\mathbb R}^n },\; \tau \in (0, T) \}.
$$
 
The explicit formula for the sharp coefficient
\begin{equation} \label{Eq_1.5}
{\mathcal C}_{p, \ell}(t)=\frac{\big |A^{-1/2}\ell \big |}
{\big \{ 2^{n}\pi^{(n+p-1)/2}\det A^{1/2} \big \}^{1/p}}\left \{ \frac{\Gamma \left (\frac{p'+1}{2} \right )} {p'^{(n+p')/2}} \int_0^t \frac{e^{p'c\tau}}{\tau ^{(n(p'-1)+p' )/2}}d\tau\right \}^{1/p'}
\end{equation} 
in the inequality
\begin{equation} \label{Eq_1.5A} 
\left | \frac{\partial u}{ \partial {\ell}}(x, t) \right |\leq {\mathcal C}_{p, \ell}(t)\nl  f \nr_{p, t}
\end{equation} 
for solutions of the Cauchy problem (\ref{NH}) is found, where $(x, t)\in {\mathbb R}^{n+1}_T$. 
As a consequence of (\ref{Eq_1.5}), we arrive at the formula for the sharp coefficient
\begin{equation} \label{Eq_1.5AB}
{\mathcal C}_p(t)=\max_{|\ell |=1}{\mathcal C}_{p, \ell}(t)
\end{equation} 
in the inequality
\begin{equation} \label{Eq_1.5ABC}
\left | \nabla_x u(x, t) \right |\leq {\mathcal C}_p(t) \nl f \nr_{p, t}\;.
\end{equation} 
For instance, 
\begin{equation} \label{Eq_1.6}
{\mathcal C}_\infty(t)=\frac{\bv A^{-1/2}\bv}{\sqrt{\pi}}\;\int_0^t
\frac{e^{c\tau}}{\sqrt{\tau}}\; d\tau\;.
\end{equation}
In the particular case  $A=aI$,  formula (\ref{Eq_1.6}) takes the form
$$
{\mathcal C}_\infty(t)=\frac{1} {\sqrt{a\pi}}\;\int_0^t
\frac{e^{c\tau}}{\sqrt{\tau}}\; d\tau\;.
$$

It can be of interest that the sharp coefficients in (\ref{Eq_1.3}) and (\ref{Eq_1.5}) do not depend on the coefficient  vector $b=(b_1, \dots, b_n)$. 

%%%%%%%%%%%%%%%%%%%%%%%%%%%%%%%%%%%%%%%%%%%%%%%%%%%%%%%%%%%%%%%%%%
\section{Explicit formulas for solutions} \label{S_2}
%%%%%%%%%%%%%%%%%%%%%%%%%%%%%%%%%%%%%%%%%%%%%%%%%%%%%%%%%%%%%%%%%%

By $(x,  y)$ we mean the inner product of the vectors $ x$ and $ y$ in ${\mathbb R}^n$.
The Schwartz class of rapidly decreasing $C^{\infty}$-functions on ${\mathbb R}^n$ will be denoted by ${\mathcal S}(\mathbb R^n)$.

Since the matrices $A$ and $A^{-1}$ are positive definite, there exist positive definite matrices $A^{1/2}$ and $A^{-1/2}$ of order $n$ such that $\big ( A^{1/2} \big )^2=A$ and $\big ( A^{-1/2} \big )^2=A^{-1}$ (e.g. \cite{LA}, sect. 2.14).

We start with the Cauchy problem (\ref{H}) for the homogeneous equation. 

The assertion below can be proved analogously to the corresponding statement for the heat equation (e.g.\cite{Hu}, Th. 5.4; \cite{OO}, Sect. 4.8; \cite{MS},  Sect. 7.4). 
In the proof of this assertion we preserve only the formal arguments leading to the representation for a solution of problem (\ref{H}).

\begin{lemma} \label{L_1} Let $\varphi \in {\mathcal S}(\mathbb R^n)$. A 
solution $u$ of problem $(\ref{H})$ is given by
\begin{equation} \label{Eq_3.2}
u(x, t)=\int_{{\mathbb R}^n}G(x-y, t)\varphi(y)dy\;,
\end{equation}
where
\begin{equation} \label{Eq_3.3}
G(x, t)= \frac{e^{ct}}{\big ( 2\sqrt{\pi t}\big )^n\det A^{1/2} } e^ {-\frac{1}{4t}\left | A^{-1/2}(x+tb) \right |^2}\;.
\end{equation}
\end{lemma}
\begin{proof} Applying the Fourier transform
\begin{equation} \label{eq_4.2}
\hat{u}(\xi, t)= \frac{1}{(2\pi)^{n/2}} \int_{\mathbb R ^n} e ^{-i( x,\xi)}u(x, t)dx
\end{equation}
to the Cauchy problem (\ref{H}), we obtain
\begin{equation} \label{Eq_4.3}
\frac{d\hat{u}}{dt}=\big \{-(A\xi, \xi)+i (b, \xi)+c \big \}\hat{u}\;,\;\;\;\;\;\hat{u}(\xi,0)=\hat{\varphi}(\xi)\;.
\end{equation}
The solution of problem (\ref{Eq_4.3}) is
\begin{equation} \label{Eq_4.4}
\hat{u}(\xi, t)=\hat{\varphi }(\xi)e^{ \{-(A\xi, \xi))+i (b, \xi)+c  \}t}\;.
\end{equation}
By the inverse Fourier transform
$$
u(x, t)=\frac{1}{(2\pi)^{n/2}}\int_{\mathbb R ^n} e ^{i( x, \xi)}\hat{u}(\xi ,t)d\xi, 
$$
we deduce from (\ref{Eq_4.4})
\begin{eqnarray} \label{Eq_4.5a}
& & u(x, t)= \frac{1}{(2\pi)^{n/2}}\int_{\mathbb R ^n} e ^{i( x,\xi)}\hat{\varphi}(\xi)e^{ \{-(A\xi, \xi))+i (b, \xi)+c  \}t}d\xi \nonumber \\
& & = \frac{e^{ct}}{(2\pi)^{n}}\int_{\mathbb R ^n} e ^{i( x,\xi)}e^{ \{-(A\xi, \xi))+i (b, \xi) \}t}  \left \{ \int_{\mathbb R ^n} e ^{-i( y, \xi)}{\varphi}(y)dy \right \} d\xi \nonumber\\
& & =\frac{e^{ct}}{(2\pi)^{n}}\int_{\mathbb R ^n} \left\{ \int_{\mathbb R ^n} e^{i( x- y+tb, \xi)}e^{-(tA\xi, \xi)} d\xi \right\} {\varphi}(y)dy.  
\end{eqnarray}
Let us denote
\begin{equation} \label{Eq_4.6}
    G(x,t)=  \frac{e^{ct}}{(2\pi)^n}\int_{\mathbb R ^n} e^{-(tA\xi, \xi)+i( x+tb, \xi)}d\xi .
\end{equation}

The known formula (e.g.\cite{VL}, Ch. 2, Sect. 9.7)
$$
\int_{\mathbb R ^n} e^{-(M\xi, \xi)+i( \zeta,\xi)}d\xi=
\frac{\pi^{n/2}}{\sqrt{\det M}}e^{-\frac{1}{4}( \zeta, M^{-1}\zeta)},
$$
where $M$ is a symmetric positive definite matrix of order $n$,
together with (\ref{Eq_4.6}) leads to
\begin{equation} \label{Eq_4.7}
   G(x,t)=  \frac{e^{ct}}{(2\sqrt{\pi t})^n\sqrt{\det A}}e^{-\frac{1}{4t}\left (A^{-1} ( x+tb) ,  x+tb \right )}\;.
\end{equation}

Since $(A^{-1}\zeta, \zeta)=(A^{-1/2}A^{-1/2}\zeta, \zeta)=
(A^{-1/2}\zeta, A^{-1/2}\zeta )=|A^{-1/2}\zeta|^2$ for any $\zeta \in {\mathbb R ^n}$ and $\det A=\big ( \det A^{1/2}\big )^2$, we can write (\ref{Eq_4.7}) as (\ref{Eq_3.3}). It follows from (\ref{Eq_4.5a}) and (\ref{Eq_4.6}) that the solution of problem (\ref{H}) can be represented as (\ref{Eq_3.2}),
where $G(x,t)$ is given by (\ref{Eq_3.3}).
\end{proof}

\begin{remark} {\bf 1.} Changing the variable 
$\xi= A^{-1/2}(x+tb)$ in the integral
$$
\int_{\mathbb R ^n} 
 e^ {-\frac{1}{4t}\left | A^{-1/2}(x+tb) \right |^2} dx\;,
$$
we arrive at the equality
\begin{equation} \label{Eq_4.8} 
\nl G(\cdot, t) \nr_{1}=\int_{\mathbb R ^n} G(x,t) dx=\frac{e^{ct}}{\big ( 2\sqrt{\pi t}\big )^n\det A^{1/2} }\int_{\mathbb R ^n} 
 e^ {-\frac{1}{4t}\left | A^{-1/2}(x+tb) \right |^2} dx=e^{ct},
\end{equation} 
which generalizes the analogous fact for the heat equation.
In particular, (\ref{Eq_0.1A}) follows from (\ref{Eq_4.8}).
\end{remark}

\medskip
The next assertion can be proved on the base of Lemma \ref{L_1}  similarly to the analogous statement for the heat 
equation (e.g.\cite{Hu}, Th. 5.5; \cite{MS},  Sect. 7.4). 

\begin{proposition} \label{P_1} Suppose that $1\leq p\leq \infty$
and $\varphi \in L^p({\mathbb R}^n)$. Define
$u: {\mathbb R}^{n+1}_T\rightarrow {\mathbb R}$
by $(\ref{Eq_3.2})$, where $G$ is given by $(\ref{Eq_3.3})$.
Then $u(x, t)$ is solution of the equation
$$
\frac{\partial u}{\partial t}= \sum_{j,k=1}^n a_{jk}\frac{\partial^2 u}{\partial x_j \partial x_k}+
\sum_{j=1}^n b_{j}\frac{\partial u}{\partial x_j}+c u
$$ 
in ${\mathbb R}^{n+1}_T$. If $1\leq p< \infty$, then 
$u(\cdot , t)\rightarrow \varphi$ in $L^p$ as $t \rightarrow 0^+$.
\end{proposition}

\begin{remark} {\bf 2.}
It is known (e.g. \cite{FR}, Ch.1, Sect. 6, 7 and 9) that  (\ref{Eq_3.2}), where $G$ is given by (\ref{Eq_3.3}), represents a unique bounded solution of the Cauchy problem (\ref{H}) and $u(x , t)\rightarrow \varphi(x)$ as $t \rightarrow 0^+$ at any $x\in {\mathbb R}^n$ under assumption that $\varphi$ belongs to $C({\mathbb R}^n)\cap L^\infty ({\mathbb R}^n)$. This fact together with the Lusin's theorem 
(see \cite{VU}, Ch. VI, Sect. 6) implies that $u(\cdot , t)\rightarrow \varphi $ almost everywhere in ${\mathbb R}^n$ as $t \rightarrow 0^+$, where $u$ is a solution of problem (\ref{H}) with $\varphi \in L^\infty({\mathbb R}^n)$.
\end{remark}

\medskip
Further, let us consider the Cauchy problem (\ref{NH}) for the nonhomogeneous equation. 

The next statement can be proved analogously to the similar assertion for the heat equation (e.g. \cite{OO}, Sect. 4.8). 
In the proof of this assertion we preserve only the formal arguments leading to the representation for a solution of problem (\ref{NH}).

\begin{lemma} \label{L_2} Let $f(\cdot, t)\in {\mathcal S}({\mathbb R}^n)$ for any $t\in [0, T]$ and let 
the quantities $C_{\alpha, m}$ in the estimates
$$
\left (1+|x|^m \right )|\partial_x^\alpha f(x, t)|\leq C_{\alpha, m}
$$
are independent of $t$ for any integer $m\geq 0$ and multiindex $\alpha$.

The solution of problem $(\ref{NH})$  is given by
\begin{equation} \label{Eq_3.2a}
u(x, t)=\int_0^t\int_{{\mathbb R}^n}G(x-y, t-\tau) f(y, \tau)dyd\tau\;,
\end{equation}
where $G$ is defined by $(\ref{Eq_3.3})$.
\end{lemma}
\begin{proof} Applying the Fourier transform to the Cauchy problem
(\ref{NH}), we obtain
\begin{equation} \label{Eq_4.3A}
\frac{d\hat{u}}{dt}=\big \{-(A\xi, \xi)+i (b, \xi)+c \big \}\hat{u}+\hat{f}(\xi, t)\;,\;\;\;\;\;\hat{u}(\xi,0)=0\;.
\end{equation}
The solution of problem (\ref{Eq_4.3A}) is
\begin{eqnarray} \label{Eq_4.4A}
\hat{u}(\xi, t)&=&e^{ \{-(A\xi, \xi))+i (b, \xi)+c  \}t}
\int_0^t \hat{f}(\xi, \tau)
e^{ \{(A\xi, \xi))-i (b, \xi)-c  \}\tau}d\tau\nonumber\\
&=&\int_0^t \hat{f}(\xi, \tau)
e^{ \{-(A\xi, \xi))+i (b, \xi)+c  \}(t-\tau)}d\tau\;.
\end{eqnarray}
By the inverse Fourier transform in (\ref{Eq_4.4A}), we have
\begin{eqnarray} \label{Eq_4.5A}
u(x, t)&=&\frac{1}{(2\pi)^{n/2}}\int_{\mathbb R ^n}\left \{\int_0^t \hat{f}(\xi, \tau)e^{ \{-(A\xi, \xi))+i (b, \xi)+c  \}(t-\tau)}d\tau \right \}e^{i(x, \xi)}d\xi\nonumber\\
&=&\frac{1}{(2\pi)^{n/2}}\int_0^t\left \{\int_{\mathbb R ^n} e^{i(x, \xi)} \hat{f}(\xi, \tau)e^{ \{-(A\xi, \xi))+i (b, \xi)+c  \}(t-\tau)} d\xi \right \}d\tau\nonumber\\
&=&\frac{1}{(2\pi)^n}\int_0^t\left \{\int_{\mathbb R ^n} 
e^{i(x, \xi)} 
\left \{ \int_{\mathbb R ^n}e^{-i(y, \xi)}f(y, \tau)dy \right \}
e^{ \{-(A\xi, \xi))+i (b, \xi)+c  \}(t-\tau)} d\xi \right \}d\tau
\nonumber\\
&=&\int_0^t \int_{\mathbb R ^n}
\left \{\frac{e^{c(t-\tau)}}{(2\pi)^n} \int_{\mathbb R ^n}e^{i(x-y+ib(t-\tau ), \xi)}
e^{ -(A\xi, \xi)(t-\tau)}d\xi\right \}f(y, \tau)dy d\tau.
\end{eqnarray}
By (\ref{Eq_4.6}), expression inside of the braces in the right-hand side of (\ref{Eq_4.5A}) is equal to $G(x-y, t-\tau)$.
So, equality (\ref{Eq_4.5A}) can be written as (\ref{Eq_3.2a}), where $G$ is given by (\ref{Eq_3.3}).
\end{proof}

\begin{remark} {\bf 3.}
It is known (e.g. \cite{FR}, Ch.1, Sect. 7 and 9) that formula (\ref{Eq_3.2a}), where $G$ is defined by (\ref{Eq_3.3}), solves
problem  (\ref{NH}) with $f\in C^\alpha\big (\overline{{\mathbb R}^{n+1}_T} \big )$.

Let $u(x, t)$ be a solution of problem (\ref{NH}) with $f \in C^\alpha\big (\overline{{\mathbb R}^{n+1}_T} \big )$ 
and $(x, t)\in {\mathbb R}^{n+1}_T$. Then, by (\ref{Eq_4.8}) and (\ref{Eq_3.2a}), we arrive at the sharp pointwise estimates
$$
|u(x, t)|\leq \frac{e^{ct}-1}{c}\nl f \nr_{\infty, t}\;,
$$
where $c\neq 0$, and
$$
|u(x, t)|\leq t\nl f \nr_{\infty, t}\;,
$$
where $c=0$.
  
The last estimate is well known for solutions of the Cauchy problem with zero initial data for the nonhomogeneous heat equation (e.g. \cite{VL}, Ch. III, Sect. 16). 
\end{remark}

%%%%%%%%%%%%%%%%%%%%%%%%%%%%%%%%%%%%%%%%%%%%%%%%%%%%%%%%%%%%%%%%%%
\section{Estimates for solutions of the homogeneous equation} \label{S_3}
%%%%%%%%%%%%%%%%%%%%%%%%%%%%%%%%%%%%%%%%%%%%%%%%%%%%%%%%%%%%%%%%%%

In this section we consider the Cauchy problem (\ref{H}).
Here we suppose that $\varphi\in L^p({\mathbb R}^n)$, where $p \in [1, \infty]$.

\setcounter{theorem}{0}
\begin{theorem} \label{T_1} Let $(x, t)$ be an arbitrary point in ${\mathbb R}^{n+1}_T$ and $u $ be solution of problem $(\ref{H})$. 
The sharp coefficient ${\mathcal K}_{p, \ell}(t)$ in inequality
$(\ref{Eq_1.3A})$ is given by $(\ref{Eq_1.3})$.

As a consequence, the sharp coefficient ${\mathcal K}_p(t)$ in inequality $(\ref{Eq_1.3ABC})$ is given by
\begin{equation} \label{Eq_3.7}
{\mathcal K}_p(t)=\frac{\bv A^{-1/2} \bv }
{\big \{ 2^{n}\pi^{(n+p-1)/2}\det A^{1/2} \big \}^{1/p}}
\left \{ \frac{\Gamma \left (\frac{p'+1}{2} \right )} {p'^{(n+p')/2}} \right \}^{1/p'}\;\frac{e^{ct}}{t^{(n+p)/(2p)}}\;.
\end{equation}

As a special case of $(\ref{Eq_3.7})$ 
one has $(\ref{Eq_1.4})$.
\end{theorem}
\begin{proof} By (\ref{Eq_3.2}), (\ref{Eq_3.3}) and (\ref{Eq_4.7}), we have
\begin{equation} \label{Eq_3.8}
\frac{\partial u}{\partial x_j}=\int_{{\mathbb R}^n}\frac{\partial}{\partial x_j}G(x-y, t)\varphi(y)dy\;,
\end{equation}
where 
\begin{equation} \label{Eq_3.8C}
G(x,t)= \frac{e^{ct}}{\big ( 2\sqrt{\pi t}\big )^n\det A^{1/2} } e^ {-\frac{1}{4t}\left | A^{-1/2}(x+tb) \right |^2}= 
\frac{e^{ct}}{(2\sqrt{\pi t})^n\det A^{1/2}}e^{-\frac{1}{4t}\left (A^{-1} ( x+tb) ,  x+tb \right )}\;.
\end{equation}
Differentiating in (\ref{Eq_3.8C}) with respect to $x_j$, $j=1,\dots , n$, we obtain
\begin{equation} \label{Eq_3.8A}
\frac{\partial}{\partial x_j}G(x-y, t)\!=-
\frac{e^{ct}}{2t\big ( 2\sqrt{\pi t}\big )^n\det A^{1/2} }\!
\left \{ A^{-1}( x\!-\!y\!+\! t b)\right \}_j  
e^ {-\frac{1}{4t}\left | A^{-1/2}(x-y+tb) \right |^2},
\end{equation}
which together with (\ref{Eq_3.8}), leads to
\begin{equation} \label{Eq_3.9}
\frac{\partial u}{ \partial { \ell}}\!=-\frac{e^{ct}}{2t\big ( 2\sqrt{\pi t}\big )^n\det A^{1/2} }
\int_{{\mathbb R}^n}\!\!\left ( A^{-1}( x\!-\!y\!+\! t b) , \ell \right )
e^ {-\frac{1}{4t}\left | A^{-1/2}(x-y+tb) \right |^2} \varphi(y)dy .
\end{equation}

Applying the H\"older inequality to the right-hand side of (\ref{Eq_3.9}), we conclude that the sharp coefficient in  estimate (\ref{Eq_1.3A})
is given by
\begin{equation} \label{Eq_3.12}
{\mathcal K}_{p, \bs\ell}(t)=\frac{e^{ct}}{2t\big ( 2\sqrt{\pi t}\big )^n\det A^{1/2} }
\left \{\int_{{\mathbb R}^n}\left |\left ( A^{-1}( x\!-\!y\!+\! t b) , \ell \right )\right |^{p'}
e^ {-\frac{p'}{4t}\left | A^{-1/2}(x-y+tb) \right |^2}dy \right \}^{\frac{1}{p'}}.
\end{equation}

Further, we introduce the new variable
$$
\xi= A^{-1/2}(x-y+tb)\;.
$$
Since $y=-A^{1/2}\xi +x+t b$, we have $dy =\det A^{1/2}d\xi$, which
together with (\ref{Eq_3.12}) leads to the following representation
$$
{\mathcal K}_{p, \ell}(t)=\frac{e^{ct}(\det A^{1/2})^{1/p'}}{2t\big ( 2\sqrt{\pi t}\big )^n\det A^{1/2} }
\left \{\int_{{\mathbb R}^n}\left |\left (A^{-1/2}\xi ,\ell \right )\right |^{p'}
e^{-\frac{p' | \xi|^2 }{4t}}d\xi \right \}^{\frac{1}{p'}}.
$$
By the  symmetricity of $A^{-1/2}$, we have
\begin{equation} \label{Eq_3.11ABCD}
{\mathcal K}_{p, \ell}(t)=\frac{e^{ct}}{2t\big ( 2\sqrt{\pi t}\big )^n(\det A^{1/2})^{1/p} }
\left \{\int_{{\mathbb R}^n}\left |\left (\xi , A^{-1/2}\ell \right )\right |^{p'}
e^{-\frac{p' |\xi|^2}{4t}}d\xi \right \}^{\frac{1}{p'}}.
\end{equation}

Passing to the spherical coordinates in (\ref{Eq_3.11ABCD}), we obtain
\begin{equation} \label{Eq_3.12A}
{\mathcal K}_{p, \ell}(t)\!=\!\frac{e^{ct}}{2t\big ( 2\sqrt{\pi t}\big )^n (\det A^{1/2})^{1/p}}  
\left \{\int_0^{\infty} \rho^{p'+n-1}e^{-\frac{p'\rho^2 }{4t}} d\rho\!\!\int_{{\mathbb S}^{n-1}} \left |
\big ( e_\sigma , A^{-1/2}\ell \big ) \right |^{p'} d\sigma
\right \}^{1/p'}\;,
\end{equation}
where $e_\sigma$ is the $n$-dimensional unit vector joining the
origin to a point $\sigma$ of the unit sphere ${\mathbb S}^{n-1}$
in ${\mathbb R}^n$.

Let $\vartheta$ be the angle between $ e_{\sigma}$ and $A^{-1/2}\ell $. We have
\begin{eqnarray} \label{Eq_3.13}
& &\int_{{\mathbb S}^{n-1}} \big |\big ( e_{\sigma}, A^{-1/2}\ell \big ) \big |^{p'} d\sigma=2\omega_{n-1}
\big |A^{-1/2}\ell \big |^{p'}\int_0^{\pi/2} 
\cos^{p'} \vartheta  \sin^{n-2}\vartheta d\vartheta\nonumber\\
& &\nonumber\\
& &=\omega_{n-1}\big |A^{-1/2}\ell \big |^{p'} B\left ( \frac{p'+1}{2}, \frac{n-1}{2} \right )
=\big |A^{-1/2}\ell \big |^{p'}\:\frac{2\pi^{(n-1)/2}\Gamma \left ( \frac{p'+1}{2} \right )}
{\Gamma \left ( \frac{n+p'}{2} \right )}\;.
\end{eqnarray}
Further, making the change of variable $\rho=\sqrt{u}$ in the integral 
$$
\int_0^{\infty}\rho^{p'+n-1} e^{-\frac{p'\rho^2 }{4t }} d\rho
$$
and using the formula (e.g. \cite{GRJ}, 3.381, item 4) 
\begin{equation} \label{Eq_3.13A}
\int_0^\infty x^{\alpha-1}e^{-\beta x}dx=\beta^{-\alpha}\Gamma(\alpha)
\end{equation}
with positive $\alpha$ and $\beta$, we obtain
\begin{equation} \label{Eq_3.14}
\int_0^{\infty} \rho^{p'+n-1} e^{-\frac{p'\rho^2 }{4t}} d\rho=\frac{1}{2}\int_0^{\infty}u^{\frac{p'+n}{2}-1}e^{-\frac{p'}
{4t }u}du=
\frac{1}{2}\left ( \frac{4t}{p'} \right )^{\frac{p'+n}{2}}{\Gamma \left ( \frac{n+p'}{2} \right )}\;.
\end{equation}
Combining (\ref{Eq_3.13}) and (\ref{Eq_3.14}) with (\ref{Eq_3.12A}), we arrive at (\ref{Eq_1.3}). 

Formula (\ref{Eq_3.7}) follows from (\ref{Eq_1.3}) and (\ref{Eq_1.3AB}).
As a particular case of (\ref{Eq_3.7}) with $p=\infty$, we obtain (\ref{Eq_1.4}).
\end{proof}

%%%%%%%%%%%%%%%%%%%%%%%%%%%%%%%%%%%%%%%%%%%%%%%%%%%%%%%%%%%%%%%%%%
\section{Estimates for solutions of the nonhomogeneous equation} \label{S_4}
%%%%%%%%%%%%%%%%%%%%%%%%%%%%%%%%%%%%%%%%%%%%%%%%%%%%%%%%%%%%%%%%%%

In this section we consider the Cauchy problem (\ref{NH}) for the nonhomogeneous equation. Here we suppose that $f\in L^p({\mathbb R}^{n+1}_T)\cap C^\alpha\big (\overline{{\mathbb R}^{n+1}_T} \big ) $, where $p>n+2$ and $\alpha \in (0, 1)$.

\begin{theorem} \label{T_2} Let $(x, t)$ be an arbitrary point in ${\mathbb R}^{n+1}_T$ and let $u $ solve  problem $(\ref{NH})$.  
The sharp coefficient ${\mathcal C}_{p,\ell}(t)$ in inequality
$(\ref{Eq_1.5A})$ is given by $(\ref{Eq_1.5})$.

As a consequence of $(\ref{Eq_1.5})$, the sharp coefficient  
${\mathcal C}_{p}(t)$ in inequality $(\ref{Eq_1.5ABC})$ is given by
\begin{equation} \label{Eq_3.17A}
{\mathcal C}_p(t)=\frac{\bv A^{-1/2} \bv }
{\big \{ 2^{n}\pi^{(n+p-1)/2}\det A^{1/2} \big \}^{1/p}}\left \{ \frac{\Gamma \left (\frac{p'+1}{2} \right )} {p'^{(n+p')/2}} \int_0^t \frac{e^{p'c\tau}}{\tau ^{(n(p'-1)+p' )/2}}d\tau\right \}^{1/p'}.
\end{equation}

As a special case of $(\ref{Eq_3.17A})$ one has $(\ref{Eq_1.6})$.
\end{theorem}
\begin{proof} By (\ref{Eq_3.8A}) and (\ref{Eq_3.2a}), we have
\begin{equation} \label{Eq_3.17CD}
\frac{\partial u }{ \partial {\ell}}=\int_0^t
\int_{{\mathbb R}^n}\big ( \nabla_x  G(x-y, t-\tau), \ell  \big )f(y, \tau)dy d\tau,
\end{equation}
where
$$
\big ( \nabla_x  G(x\!-\!y, t\!-\!\tau \!), \ell  \big )\!=-c_n\frac{e^{c(t-\tau)}}{(t-\tau)^{(n+2)/2} }
\!\left ( A^{-1}( x\!-\!y\!+\! (t-\tau) b) , \ell \right )
e^ {-\frac{1}{4(t-\tau)}\left | A^{-1/2}(x-y+(t-\tau)b) \right |^2}.
$$
Here
\begin{equation} \label{Eq_3.18}
c_n=\frac{1}{2^{n+1}\pi^{n/2}\det A^{1/2}}\;.
\end{equation}

Applying the H\"older inequality to the right-hand side of (\ref{Eq_3.17CD}), we conclude that the sharp coefficient 
${\mathcal C}_{p, \ell }(t)$ in the estimate (\ref{Eq_1.5A}) is given by
$$
{\mathcal C}_{p,\ell }(t)\!=\!\! c_n\!\!\left \{ \int_0^t\int_{{\mathbb R}^n}\!\frac{e^{p'c(t\!-\!\tau)}}{\big ( t-\tau\big )^{(n+2)p'/2} }
\big |\!\left ( A^{-1}( x\!-\!y\!+\! (t-\tau) b) , \ell \right )\!\big |^{p'} e^{-\frac{p'\left | A^{-1/2}(x-y+(t-\tau)b) \right |^2}{4 (t-\tau)}}
dy d\tau \!\right \}^{\frac{1}{p'}}\!\!. 
$$

Changing the variable $\tau=t-\eta$, we rewrite the previous representation for ${\mathcal C}_{p,\ell }(t)$ as
\begin{equation} \label{Eq_3.19A}
{\mathcal C}_{p,\ell }(t)\!=\!\! c_n\!\!\left \{ \int_0^t\int_{{\mathbb R}^n}\frac{e^{p'c\eta}}{\eta^{(n+2)p'/2} }
\big |\!\left ( A^{-1}( x\!-\!y\!+\! \eta b) , \ell \right )\!\big |^{p'} e^{-\frac{p'\left | A^{-1/2}(x-y+\eta b) \right |^2}{4 \eta }}
dy d\eta \!\right \}^{\frac{1}{p'}}. 
\end{equation}
Now, we introduce the new variable  
$$ 
\xi=A^{-1/2}(x-y)+\eta b
$$ 
in the inner integral in the right-hand side of (\ref{Eq_3.19A}). By the symmetricity of $A^{-1/2}$ and by the relation $dy = \det A^{1/2}d\xi$, we arrive at
\begin{equation} \label{Eq_3.20}
{\mathcal C}_{p,\bs\ell }(t)\!=\!\!c_n\left (\det A^{1/2} \right )^{1/p'} 
\left \{ \int_0^t\int_{{\mathbb R}^n}\frac{e^{cp'\eta }} {\eta^{(n+2)p'/2} }
\big |\big (\xi , A^{-1/2}\ell \big )\big |^{p'}
 e^{-\frac{p' | \xi |^2 }{4\eta}}d\xi
 d\eta \!\right \}^{\frac{1}{p'}}\!\!. 
\end{equation}

Passing to the spherical coordinates in (\ref{Eq_3.20}), we obtain 
$$
{\mathcal C}_{p, \bs\ell}(t)\!=
\!c_n\!\left (\det A^{1/2} \right )^{1/p'} 
 \left \{\! \int_0^t\!\!\frac{e^{cp'\eta }} {\eta^{(n+2)p'/2} } \! \left (\int_0^\infty \!\!
\rho^{n+p'-1}  e^{-\frac{p' \rho^2}{4\eta }}d\rho \right )\!d\eta\!\!
\int_{{\mathbb S}^{n-1}}\left | \big ( e_{\sigma}, A^{-1/2}\ell \big )\right |^{p'} d\sigma \! \right \}^{\frac{1}{p'}}\!\!.
$$ 
The last equality together with (\ref{Eq_3.13}), (\ref{Eq_3.14}) and (\ref{Eq_3.18}) leads to
\begin{equation} \label{Eq_3.21}
{\mathcal C}_{p, \ell}(t)\!=\frac{\big |A^{-1/2}\ell \big |}
{\big \{ 2^{n}\pi^{(n+p-1)/2}\det A^{1/2} \big \}^{1/p}}\left \{ \frac{\Gamma \left (\frac{p'+1}{2} \right )} {p'^{(n+p')/2}} \int_0^t \frac{e^{p'c\tau}}{\tau ^{(n(p'-1)+p' )/2}}d\tau\right \}^{1/p'},
\end{equation}
where the integral converges for $p>n+2$. Thus, formula (\ref{Eq_1.5}) is proved.

Equality  (\ref{Eq_3.17A}) follows from (\ref{Eq_3.21}) and (\ref{Eq_1.5AB}). 
Putting $p=\infty $ in (\ref{Eq_3.17A}), we arrive at (\ref{Eq_1.6}). 
\end{proof}

\medskip
{\bf Acknowledgement.} The publication has been prepared with the support of the "RUDN University Program 5-100".

%%%%%%%%%%%%%%%%%%%%%%%%%%%%%%%%%%%%%%%%%%%%%%%%%%%%%%%%%%%%%%%%%%

%%%%%%%%%%%%%%%%%%%%%%%%%%%%%%%%%%%%%%%%%%%%%%%%%%%%%%%%%%%%%%%%%%%%%%%%
\end{document}